\documentclass[a4paper,10pt]{amsart}
\usepackage[utf8]{inputenc}
\usepackage{amsxtra}
\usepackage{amsopn}
\usepackage{amsmath,amsthm,amssymb}
\usepackage{amscd}
\usepackage{amsfonts}
\usepackage{latexsym}
\usepackage{verbatim}
\usepackage{color}

\theoremstyle{plain}
\newtheorem{theorem}{Theorem}[section]
\newtheorem*{theorem*}{Theorem}
\newtheorem{definition}[theorem]{Definition}
\newtheorem{lemma}[theorem]{Lemma}
\newtheorem{prop}[theorem]{Proposition}
\newtheorem{cor}[theorem]{Corollary}

\newtheorem{ex}[theorem]{Example}
\newtheorem*{mt*}{Main Theorem}


\let\phi\varphi
\newcommand\C{{\mathbb C}}

\newcommand\R{{\mathbb R}}

\newcommand{\del}{\partial}

\newcommand{\delbar}{\overline{\del}}

\title[$p$-K\"ahler and balanced structures on nilmanifolds]{$p$-K\"ahler and balanced structures on nilmanifolds with nilpotent complex structures}
\author[Tommaso Sferruzza and Nicoletta Tardini]{Tommaso Sferruzza and  Nicoletta Tardini}
\address{Dipartimento di Scienze Matematiche, Fisiche e Informatiche\\
Plesso Matematico e Informatico,
Universit\`{a} di Parma\\
Parco Area delle Scienze 53/A, 43124 \\
Parma, Italy}
\email{tommaso.sferruzza@unipr.it}
\email{nicoletta.tardini@gmail.com}
\email{nicoletta.tardini@unipr.it}

\thanks{\newline 
The first author is partially supported by GNSAGA of INdAM.
The second author is partially supported by GNSAGA of INdAM and has financially been supported by the Programme ``FIL-Quota Incentivante'' of University of Parma and co-sponsored by Fondazione Cariparma.}
\begin{document}

\begin{abstract}
Let $(X,J)$ be a nilmanifold with a left-invariant nilpotent complex structure.
We study the existence of $p$-K\"ahler structures (which include K\"ahler and balanced metrics) on $X$. More precisely, we determine an optimal $p$ such that there are no $p$-K\"ahler structures on $X$.
Finally, we show that, contrarily to the K\"ahler case, on compact complex manifolds there is no relation between the existence of balanced metrics and the degeneracy step of the Fr\"olicher spectral sequence. More precisely, on balanced manifolds the degeneracy step can be arbitrarily large.
\end{abstract}
\maketitle

\section{Introduction}

Let $(X,J)$ be a compact complex manifold of complex dimension $n$ and let $g$ be an Hermitian metric on $(X,J)$ with fundamental form $\omega$. Then, it is well known that if $\omega$ is a K\"ahler metric, i.e., $d\omega=0$ then the Fr\"olicher spectral sequence degenerates at the first step. However, if there are no K\"ahler metrics on $(X,J)$ then the degeneracy step might be higher than one, as first shown in \cite{Kodaira64}. For examples of higher degeneration steps, one can refer to \cite{cordero_fernandez_ugarte_gray}. It is therefore natural to ask whether there are weaker metric conditions that impose restrictions on the degeneration of the Fr\"olicher spectral sequence. In fact, in \cite[Conjecture 1.3]{popovici} it is conjectured that if there exists an SKT metric on $(X,J)$, namely an Hermitian metric $\omega$ such that $\del\delbar\omega=0$, the degeneracy step is $2$. Another important class of metrics in non-K\"ahler geometry is provided by \emph{balanced metrics}, that are Hermitian metrics satisfying $d\omega^{n-1}=0$. Clearly, on compact complex surfaces they coincide with K\"ahler metrics, but in higher dimension they are a larger class of non necessarily K\"ahler metrics. In \cite{popovici14} it was shown that the existence of a balanced metric does not imply the degeneration at the first step, see also \cite{ceballos_otal_ugarte_villacampa}. Here, we show that in fact on balanced manifolds the degeneracy step can be arbitrarily large. More precisely, we prove the following (see Theorem \ref{thm:main} and Corollary \ref{cor:balanced-degeneration})

\begin{theorem*}
For every $n\geq 2$, there exists a nilmanifold of complex dimension $4n-2$ that admits balanced metrics and such that the Fr\"olicher spectral sequence does not degenerate at step $n$.
\end{theorem*}

Such nilmanifolds were first constructed in \cite{bigalke-rollenske}, where the degeneration of the Fr\"olicher spectral sequence was studied. Here we investigate the existence of special structures on such manifolds. In particular, we show that, other than being balanced, they do not admit any SKT and locally conformally K\"ahler metrics. Moreover, in Theorem \ref{thm:br-p-kahler} we show that they do not admit any $p$-K\"ahler structure for $1\leq p<4n-3$.\\
More generally, we study in Section \ref{section:p-kahler-nilmanifolds} the existence of $p$-K\"ahler forms on nilmanifolds with nilpotent complex structures. We recall that, for $1\leq p\leq n$, a $p$-\emph{K\"ahler} form $\Omega$ on a complex manifold $X$ of dimension $n$ is a $d$-closed real transverse $(p,p)$-form. In particular, $1$-K\"ahler forms coincide with K\"ahler metrics and $(n-1)$-K\"ahler forms coincide with balanced metrics. They were introduced in \cite{alessandrini-andreatta}.
Here, we determine an optimal $p$ such that there are no $p$-K\"ahler structures on nilmanifolds with nilpotent complex structure. More precisely, we prove the following (see Theorem \ref{thm:no-n-k-kahler})

\begin{theorem*}
Let $X=\Gamma\backslash G$ be a nilmanifold of complex dimension $n$ endowed with a left-invariant nilpotent complex structure $J$. 
Let $\left\lbrace\varphi^i\right\rbrace_{i=1,\cdots,n}$ be a co-frame of left-invariant $(1,0)$-forms satisfying, for $i=1,\cdots,n$
$$
d\varphi^i\in \Lambda^2\left\langle\varphi^1,\cdots,\varphi^{i-1},\bar\varphi^1,
\cdots,\bar\varphi^{i-1}\right\rangle,
$$
and let $k$ be the index such that
$$
d\varphi^i=0 \quad\text{for}\quad i=1,\cdots,k 
\quad\text{and}\quad d\varphi^{i}\neq 0 \quad\text{for}\quad i=k+1,\cdots,n. 
$$
Then, there are no $(n-k)$-K\"ahler forms on $X$.
\end{theorem*}
In particular, for $k=1$ we obtain the conclusion for balanced metrics. Notice that, for complex dimension $3$, this result is compatible with the classification of balanced structures on $6$-dimensional nilmanifolds endowed with nilpotent complex structures proved in \cite{ugarte}.\\
We recall that nilpotent complex structures include abelian and bi-invariant complex structures. We further observe that with this degree of generality the index $p=n-k$ in Theorem \ref{thm:no-n-k-kahler}
 is optimal. Indeed, we construct in examples \ref{ex:1} and \ref{ex:2} two   $2$-step nilmanifolds with left-invariant abelian complex structures that admit a $(n-k-1)$-K\"ahler form and a $(n-k+1)$-K\"ahler form.\\
\vspace{0.2cm}

\emph{Acknowledgement.} The authors would like to kindly thank Adriano Tomassini for many useful discussions and suggestions in the preparation of the paper. The authors would also like to thank Anna Fino for pointing out the reference \cite{ugarte}.

\section{$p$-K\"ahler structures on nilmanifolds with nilpotent complex structures}\label{section:p-kahler-nilmanifolds}

In this Section we are going to discuss the existence of special structures, called $p$-K\"ahler, on nilmanifolds with nilpotent complex structures. These structures include K\"ahler and balanced metrics.

\subsection{$p$-K\"ahler structures} 

Let us start by recalling some definitions that will be used in the following.\\
Let $(X,J)$ be a complex manifold of complex dimension $n$, i.e., the datum of a differentiable manifold $X$ of real dimension $2n$ and an integrable almost complex structure $J$. 
By extending $J$ to, respectively, the complexified tangent bundle $T_{\C}X:=TX\otimes \C$, and complexified cotangent bundle $T_{\C}^*X:=T^*X\otimes \C$, we obtain the following direct sum decompositions in terms of the respective $\pm i$-eigenspaces, i.e.,
\begin{align*}
&T_{\C}X=T^{1,0}X\oplus T^{0,1}X\,,\\
&T_{\C}^*X=(T^{1,0}X)^*\oplus (T^{0,1}X)^*,
\end{align*}
and, by considering the exterior powers of $T^*_{\C}X$, we obtain
\[
\textstyle\bigwedge_{\C}^kX:=\bigwedge^{k}(T_{\C}^*X)=\bigoplus_{p+q=k}\bigwedge^{p,q}X,
\]
where each $\bigwedge^{p,q}X:=\bigwedge^p (T^{1,0}X)^{\ast}\otimes\bigwedge^{q}(T^{0,1}X)^{\ast}$ is the bundle of \emph{$(p,q)$-forms} on $X$. We will denote with $A^k(X, \C)$ and $A^{p,q}(X)$ the spaces of smooth sections of the bundles $\bigwedge_{\C}^kX$ and $\bigwedge^{p,q}X$ respectively.\\
We recall the following pointwise definitions. Fix $x\in X$. We denote with $T_xX$ the tangent space of $X$ at $x$ and with $T_x^*X$ its dual. If we consider $\{\phi^i\}_{i=1}^n$ a basis of $\bigwedge^{1,0}(T^*_xX\otimes\C)$, then a basis of $\bigwedge^{p,q}(T^*_xX\otimes \C)$ is given by
\[
\{\phi^{i_1}\wedge\dots\wedge\phi^{i_p}\wedge\overline{\phi}^{j_1}\wedge\dots\wedge\overline{\phi}^{j_q}:\, 1\leq i_1<\dots <i_p\leq n,\, 1\leq j_1<\dots <j_q\leq n\}.
\]
Let us denote the constant $\sigma_p:=i^{p^{2}}2^{-p}$ and the space of \emph{real $(p,p)$-forms} by
\[
\textstyle\bigwedge_{\R}^{p,p}(T^*_xX)=\{\alpha\in\bigwedge^{p,p}(T^*_xX\otimes\C): \overline{\alpha}=\alpha\}.
\]
Then, a basis for $\bigwedge_{\R}^{p,p}(T^*_xX)$ is given by
\[
\{\sigma_p\phi^{i_1}\wedge\dots\wedge\phi^{i_p}\wedge\overline{\phi}^{i_1}\wedge\dots\wedge\overline{\phi}^{i_p}:1\leq i_1<\dots < i_p\leq n\}.
\]
We note that the $(n,n)$-form
\[
vol=(\frac{i}{2}\phi^1\wedge\overline{\phi}^1)\wedge\dots\wedge(\frac{i}{2}\phi^n\wedge\overline{\phi}^n)=
\sigma_n \phi^1\wedge\dots\wedge\phi^n\wedge \overline{\phi}^1\wedge\dots\wedge\overline{\phi}^n
\]
is indeed real and, thus, it defines a volume.\\
We say that a real $(n,n)$-form $\alpha$ is \emph{positive}, respectively \emph{strictly positive}, if it holds that
\[
\alpha=c\cdot vol,
\]
with $c\geq 0$, respectively $c>0$. 
A $(p,0)$-form $\alpha$ is said to be \emph{simple} if
\[
\alpha=\alpha^1\wedge\dots\wedge\alpha^p,
\]
for $\alpha^i\in\bigwedge^{1,0}(T^*_xX\otimes \C)$.\\
A real $(p,p)$-form $\Omega\in \bigwedge_{\R}^{p,p}(T^*_xX)$ is said to be \emph{transverse} if the $(n,n)$-form
\[
\Omega\wedge\sigma_{n-p}\alpha\wedge\overline{\alpha}
\]
is strictly positive, for every non-zero simple form $\alpha\in\bigwedge^{n-p,0}(T^*_xX\otimes \C)).$\\

We recall the following
\begin{definition}
Let $(X,J)$ be a complex manifold of complex dimension $n$ and let $1\leq p\leq n$.
A $p$-\emph{K\"ahler} form $\Omega$ on $X$ is a $d$-closed real transverse $(p,p)$-form, namely $d\Omega=0$ and, at every point $x\in X$, $\Omega_x\in \bigwedge_{\R}^{p,p}(T^*_xX)$ is transverse.
\end{definition}
Notice that, by definition, for $p=1$ we obtain K\"ahler metrics and for $p=n-1$ we obtain balanced metrics. Indeed, if $\Omega$ is a $(n-1)$-K\"ahler form then by \cite{michelson} there exists an Hermitian metric $\omega$ on $X$ such that $\omega^{n-1}=\Omega$, so in particular $d\omega^{n-1}=0$. We also point out that for $1<p<n-1$,
 $p$-K\"ahler forms have no metric meaning (cf. \cite[Proposition 2.1]{alessandrini-bassanelli}). However, notice that if $\Omega$ is a transverse $(p, p)$-form on $X$, and $Y$ is a $p$-dimensional complex submanifold of $X$,then $\Omega_{|Y}$ is a volume form on $Y$.
 Such structures have been originally introduced and studied in \cite{alessandrini-andreatta,alessandrini-andreatta-errata,alessandrini-bassanelli}.
Recently, their behaviour under small deformations of the complex structure has been studied in \cite{rao_wan_zhao}. In \cite{hind-medori-tomassini} such forms have been extended to non-integrable almost-complex manifolds and we recall here the following lemma that provides an obstruction to their existence, see \cite[Proposition 3.4]{hind-medori-tomassini}.
\begin{lemma}[\cite{hind-medori-tomassini}]\label{lemma:hind-medori-tomassini}
 Let $(X,J)$ be a compact complex manifold of complex dimension $n$. Suppose that there exists a non-closed $(2n-2p-1)$-form $\eta$ such that the $(n-p,n-p)$-component of $d\eta$ satisfies
$$
(d\eta)^{n-p,n-p}=\sum_k c_k\psi_k\wedge\bar\psi_k
$$
where the $\psi_k$ are simple $(n-p,0)$-forms and the $c_k$ have the same sign. Then, $(X,J)$ does not admit a $p$-K\"ahler form.
\end{lemma}

We will use this lemma to prove the non-existence of a $p$-K\"ahler form (for suitable $p$) on nilmanifolds with nilpotent complex structures.\\

\subsection{Nilmanifolds with nilpotent complex structure}

Let $X=\Gamma\backslash G$ be a $2n$-dimensional nilmanifold, namely $G$ is a connected, simply-connected nilpotent Lie group and $\Gamma$ is a lattice in $G$. Let $J$ be a left-invariant complex structure on $X$. We denote with $\mathfrak{g}$ the Lie algebra of $G$. Then, $J$ is said to be \emph{nilpotent} if the ascending series $\left\lbrace \mathfrak{g}_i^J\right\rbrace_{i\geq 0}$ defined by
$$
\mathfrak{g}_0^J=0,\qquad
\mathfrak{g}_i^J=\left\lbrace X\in\mathfrak{g}\,|\,
[X,\mathfrak{g}]\subseteq \mathfrak{g}_{i-1}^J,\,
[JX,\mathfrak{g}]\subseteq \mathfrak{g}_{i-1}^J
\right\rbrace
$$
satisfies $\mathfrak{g}_k^J=\mathfrak{g}$ for some $k>0$.

Then, by \cite[Theorem 2]{cordero-fernandez-gray-ugarte-1} (cf. also \cite[Theorem 12]{cordero-fernandez-gray-ugarte-2}), $J$ is nilpotent if and only if there exists a co-frame of left-invariant $(1,0)$-forms $\left\lbrace\varphi^i\right\rbrace_{i=1,\cdots,n}$ satisfying, for $i=1,\cdots,n$
$$
d\varphi^i\in \Lambda^2\left\langle\varphi^1,\cdots,\varphi^{i-1},\bar\varphi^1,
\cdots,\bar\varphi^{i-1}\right\rangle.
$$
From now on we will abbreviate e.g., $\varphi^{ij\bar k}:=\varphi^i\wedge\varphi^j\wedge\bar\varphi^k$ and so on.
Now we prove the following
\begin{theorem}\label{thm:no-n-k-kahler}
Let $X=\Gamma\backslash G$ be a nilmanifold of complex dimension $n$ endowed with a left-invariant nilpotent complex structure $J$. With the above notations, let $k$ be the index such that
$$
d\varphi^i=0 \quad\text{for}\quad i=1,\cdots,k 
\quad\text{and}\quad d\varphi^{i}\neq 0 \quad\text{for}\quad i=k+1,\cdots,n. 
$$
Then, there are no $(n-k)$-K\"ahler forms on $X$.
\end{theorem}
\begin{proof}
In order to prove the result we will exhibit a $(2k-1)$-form $\eta$ satisfying the hypothesis of Lemma \ref{lemma:hind-medori-tomassini}.\\
Since $d\varphi^{k+1}\neq 0$ then at least one between $\del \varphi^{k+1}$ and $\delbar \varphi^{k+1}$ is different from $0$. Suppose now that $\delbar \varphi^{k+1}\neq 0$. We will deal later with the other case.
Since $J$ is nilpotent,
$$
\delbar\varphi^{k+1}=\sum_{l,m=1}^k C_{l\bar m}\,\varphi^{l\bar m}\neq 0
$$
for some constants $C_{l\bar m}$.
Hence, we fix two indices $i,j\leq k$ such that $C_{i\bar j}\neq 0$.\\
We define the following $(2k-1)$-form
$$
\eta=\varphi^{1\cdots \hat i\cdots k+1\, \bar 1\cdots \hat{\bar j}\cdots \bar k},
$$
where $\hat{\phi}^i$ and $\hat{\overline{\phi}^j}$ mean that we are removing the forms $\phi^i$ and $\overline{\phi}^j$ from $\eta$.\\
By the structure equations, since $d\varphi^i=0$ for $i=1,\cdots,k$ and $J$ is nilpotent,
$$
d\eta=\pm C_{i\bar j}\,\varphi^{1\cdots k\, \bar 1\cdots \bar k}
$$ 
hence $\eta$ satisfies the hypothesis of Lemma \ref{lemma:hind-medori-tomassini} and so there is no $(n-k)$-K\"ahler structure on $X$.\\
On the other side, suppose that $\delbar \varphi^{k+1}=0$ and $\del \varphi^{k+1}\neq 0$.\\
Since $J$ is nilpotent,
$$
\del\varphi^{k+1}=\sum_{l,m=1, l<m}^k A_{l\,m}\,\varphi^{l\, m}\neq 0
$$
for some constants $A_{l\,m}$.
Hence, we fix two indices $i<j\leq k$ such that $A_{i\, j}\neq 0$.\\
We define the following $(2k-1)$-form
$$
\eta=\varphi^{1\cdots \hat i\cdots \hat{j}\cdots k+1\, \bar 1\cdots \bar k}.
$$
By the structure equations, since $d\varphi^i=0$ for $i=1,\cdots,k$ and $J$ is nilpotent,
$$
d\eta=\pm A_{i\,j}\varphi^{1\cdots k\, \bar 1\cdots \bar k}
$$ 
hence $\eta$ satisfies the hypothesis of Lemma \ref{lemma:hind-medori-tomassini} and so there is no $(n-k)$-K\"ahler structure on $X$.\\
\end{proof}

As a Corollary for $k=1$ one gets immediately
\begin{cor}\label{cor:no-balanced}
Let $X=\Gamma\backslash G$ be a nilmanifold of complex dimension $n$ endowed with a left-invariant nilpotent complex structure $J$, with co-frame of $(1,0)$-forms $\left\lbrace\varphi^i\right\rbrace_{i=1,\cdots,n}$ satisfying the following structure equations,
$$
d\varphi^1=0 
\quad\text{and}\quad d\varphi^{i}\neq 0 \quad\text{for}\quad i=2,\cdots,n. 
$$
Then, there are no balanced metrics on $X$.
\end{cor}

We notice that there are large classes of complex nilmanifolds where Theorem \ref{thm:no-n-k-kahler} can be applied. 
For instance, if $J$ is \emph{abelian}, namely $[Jx,Jy]=[x,y]$ for every $x,y\in\mathfrak{g}$, or \emph{bi-invariant}, namely $J[x,y]=[Jx,y]$ for every $x,y\in\mathfrak{g}$, then it is nilpotent (cf. \cite{salamon}). Moreover, by \cite{ornea-otiman-stanciu} if $(X,J)$ is a $2$-step nilmanifold with left-invariant complex structure and $J$-invariant center, then $J$ is nilpotent.\\
We now show that $p=n-k$ in Theorem \ref{thm:no-n-k-kahler}
 is optimal. Indeed, we will show now two examples of $2$-step nilmanifolds with left-invariant abelian complex structures that admit a $(n-k-1)$-K\"ahler form and a $(n-k+1)$-K\"ahler form.

\begin{ex}\label{ex:1}
Let $X$ be the $6$-dimensional $2$-step nilmanifold with abelian complex structure defined by the following structure equations
$$
d\varphi^1=d\varphi^2=0,\qquad d\varphi^3=\varphi^{1\bar2}
$$
where $\left\lbrace\varphi^i\right\rbrace_{i=1,2,3}$ is a co-frame of $(1,0)$-forms.\\
With the previous notations we have $n=3$ and $k=2$. So, by Theorem \ref{thm:no-n-k-kahler} there are no $1$-K\"ahler forms on $X$. Of course, this was already known since on non-toral nilmanifolds there are no K\"ahler metrics.\\
Now, we show that there exists a $2$-K\"ahler form on $X$, namely a $(n-k+1)$-K\"ahler form.\\
Let
$$
\Omega:=-\varphi^{1\bar1 2\bar2}-\varphi^{1\bar1 3\bar3}-\varphi^{2\bar2 3\bar3}\,.
$$
Then, $\Omega$ is a real transverse $(2,2)$-form and by the structure equations
$$
d\Omega=0\,.
$$
Hence, $\Omega$ is a $2$-K\"ahler form on $X$. In particular, there exists a balanced metric $\omega$ on $X$ such that $\omega^2=\Omega$.
In fact, it is easy to see that
$$
\omega=i\phi^{1\bar 1}+i\phi^{2\bar 2}+i\phi^{3\bar 3}.
$$
\end{ex}

\begin{ex}\label{ex:2}
Let $X$ be the $8$-dimensional $2$-step nilmanifold with abelian complex structure defined by the following structure equations
$$
d\varphi^1=0,\qquad d\varphi^2=d\varphi^3=d\varphi^4=\varphi^{1\bar1}
$$
where $\left\lbrace\varphi^i\right\rbrace_{i=1,2,3,4}$ is a co-frame of $(1,0)$-forms.\\
With the previous notations we have $n=4$ and $k=1$. So, by Corollary \ref{cor:no-balanced} there are no balanced metrics on $X$.\\
Now, we show that there exists a $2$-K\"ahler form on $X$, namely a $(n-k-1)$-K\"ahler form.\\
Let
$$
\Omega:=-\varphi^{1\bar1 2\bar2}-\varphi^{1\bar1 3\bar3}-\varphi^{1\bar1 4\bar4}-
\varphi^{2\bar2 3\bar3}-\varphi^{2\bar2 4\bar4}-
\varphi^{3\bar3 4\bar4 }+
$$
$$
+
\varphi^{2\bar2 3\bar4}+
\varphi^{2\bar2 4\bar3 }+
\varphi^{2\bar4 3\bar3}+
\varphi^{4\bar2 3\bar3}+
\varphi^{2\bar3 4\bar4}+
\varphi^{3\bar2 4\bar4}
\,.
$$
Then, $\Omega$ is a real transverse $(2,2)$-form and by the structure equations one can see directly that
$$
d\Omega=0\,.
$$
Hence, $\Omega$ is a $2$-K\"ahler form on $X$.
\end{ex}

\section{Special Hermitian metrics on the Bigalke and Rollenske's nilmanifolds}

In this Section we are going to discuss the existence of special Hermitian metrics and $p$-K\"ahler forms on 
the $2$-step nilmanifolds with nilpotent complex structure constructed by Bigalke and Rollenske in \cite{bigalke-rollenske}. In particular, 
for every $n\geq 2$, these $(4n-2)$-dimensional compact complex manifolds are  such that the Fr\"olicher spectral sequence does not degenerate at the $E_n$ term.\\
We start by recalling the construction.
Fix $n\geq 2$ and let $G_n$ be the real nilpotent subgroup of $GL(2n+2,\mathbb{C})$ consisting of the matrices of the form
$$
\left(
\begin{array}{ccccccccccc} 
1& 0 &  &  &  &  & \cdots & & 0 &\bar y_1 & w_1\\
  & 1 &0 & \cdots &0 &\bar z_1 &-x_1 &0 & \cdots & 0 & w_2\\
  &    &\ddots & & & & \ddots & & & \vdots &\vdots\\
   & & & 1& 0& \cdots&0 & \bar z_{n-1}& -x_{n-1} &0& w_n\\
   & & &  &  1& 0& & \cdots & & 0 & y_1\\
    & & & & & \ddots & & & & \vdots &\vdots\\
    & & & & &  & & & &  & \\
    & & & & &  & & & &  & \\
    & & & & &  &  & \ddots & & \vdots & \vdots \\
    & & & & &  & & & 1 & 0 & y_n\\
     & & & & &  & & & & 1 & z_1 \\
    & & & & &  & & & &  & 1 \\
\end{array}
\right).
$$
with
$x_1, \ldots , x_{n-1}, y_1, \ldots , y_n, z_1, \ldots , z_{n-1}, w_1,\ldots, w_n\in\mathbb{C}.
$
\\
Let $\Gamma$ be the subgroup of $G_n$ consisting of the matrices of the same form and entries in $\mathbb{Z}[i]$. Then, $\Gamma$ is a lattice in $G_n$ and the quotient $X^{4n-2}:=\Gamma\backslash G_n$ is a compact $(4n-2)$-dimensional $2$-step nilmanifold with a left-invariant complex structure.\\
A global co-frame of left-invariant $(1,0)$-forms is given by
$$
dx_1,\ldots,dx_{n-1},dy_1,\ldots,dy_n,dz_1,\ldots,dz_{n-1},\omega_1,\ldots,\omega_n
$$
where
$$
\omega_1 = d w_1 -\bar y_1 d z_1 ,\quad
\omega_k= dw_k - \bar z_{k-1}dy_{k-1} + x_{k-1}dy_k \quad (k = 2,\ldots,n).
$$
The structure equations become
$$
d(dx_j)=d(dz_j)=0,\qquad (j =1,\ldots,n-1) 
$$
$$
d(dy_j)=0,\qquad (j =1,\ldots,n)
$$
$$
\del\omega_1=0,\quad 
\delbar\omega_1=dz_1\wedge d\bar y_1
$$
$$
\del\omega_j=dx_{j-1}\wedge dy_j,\quad
 \delbar\omega_j=dy_{j-1}\wedge d\bar z_{j-1}
 \qquad (j =2,\ldots,n)
$$
In \cite{bigalke-rollenske} the authors show that that the Fr\"olicher spectral sequence of $X^{4n-2}$ has non-vanishing differential $d_n$, namely the Fr\"olicher spectral sequence does not degenerate at the $E_n$ term.

We now rename the forms $dx_j$, $dy_j$, $dz_j$, and $\omega_j$ by considering the basis of $(1,0)$-forms $\{\phi^j\}_{j=1}^{4n-2}$, defined as follows
\begin{align*}
\phi^j:=\begin{cases}
dx_j,\quad &1\leq j< n\\
dy_j, \quad &n\leq j< 2n\\
dz_j, \quad &2n\leq j < 3n-1\\
\omega_j, \quad &3n-1\leq j\leq  4n-2.
\end{cases}
\end{align*}
As a result, the structure equations become
\begin{align}\label{eq:struct_eq_phi}
d\phi^j=
\begin{cases}
0, \quad &1\leq j<3n-1\\
\phi^{2n}\wedge\overline{\phi}^{n},\quad &j=3n-1\\
\phi^{j-3n+1}\wedge\phi^{j-2n+1}+\phi^{j-2n}\wedge\overline{\phi}^{j-n}, \quad &3n\leq j\leq 4n-2,
\end{cases}
\end{align}
or, more precisely,
\begin{align}\label{eq:struct_eq_phi_del}
\del\phi^j=
\begin{cases}
0, \quad &1\leq j\leq 3n-1\\
\phi^{j-3n+1}\wedge\phi^{j-2n+1}, \quad &3n\leq j\leq 4n-2,
\end{cases}
\end{align}
and
\begin{align}\label{eq:struct_eq_phi_delbar}
\delbar\phi^j=
\begin{cases}
0, \quad &1\leq j<3n-1\\
\phi^{2n}\wedge\overline{\phi}^{n},\quad &j=3n-1\\
\phi^{j-2n}\wedge\overline{\phi}^{j-n}, \quad &3n\leq j\leq 4n-2.
\end{cases}
\end{align}

\subsection{Special Hermitian metrics}

Now we study the existence of special Hermitian metrics on Bigalke and Rollenske's nilmanifolds. In particular, one can apply Theorem \ref{thm:no-n-k-kahler} and get immediately the following

\begin{prop}
For every $n\geq 2$ the Bigalke and Rollenske's nilmanifold $X^{4n-2}$ does not admit any $n$-K\"ahler form.
\end{prop}
In fact, we can show more, namely there are no $p$-K\"ahler forms except for balanced metrics.\\
More precisely, we start by proving the following

\begin{theorem}\label{thm:br-p-kahler}
For every $n\geq 2$ the Bigalke and Rollenske's nilmanifold $X^{4n-2}$ does not admit any $p$-K\"ahler form for $1\leq p< 4n-3$.
\end{theorem}

\begin{proof}
We will show that on any Bigalke and Rollenske's manifold $X^{4n-2}$, for every fixed $p$, with $1\leq p< 4n-3$, we can construct a non closed $(8n-2p-5)$-form $\eta_p$ such that the $(4n-2-p,4n-2-p)$-component of $d\eta_p$ satisfies
\begin{equation}\label{eq:hind_medori_tomassini}
(d\eta_p)^{(4n-2-p,4n-2-p)}=\epsilon_p\psi_p\wedge\overline{\psi_p},
\end{equation}
with $\psi_p\in A^{4n-2-p,0}(X^{4n-2})$ a simple form and $\epsilon_p\in\left\lbrace-1,1\right\rbrace$. By Lemma \ref{lemma:hind-medori-tomassini}, this will assure that there exists no $p$-K\"ahler form on $X^{4n-2}$.\\
Let us consider separately the cases\\
$(i)$ $1\leq p <n$;\\
$(ii)$ $n\leq p\leq 4n-2$.\\
Before doing so, we remark that, by structure equations (\ref{eq:struct_eq_phi}), the index $j$ such that every term of the expression of $d\phi^j$ contains forms with the highest indices, is $j=4n-2$. Such expression is
\[
d\phi^{4n-2}=\phi^{n-1}\wedge\phi^{2n-1}+\phi^{2n-2}\wedge\overline{\phi}^{3n-2},
\]
whereas, in general, we have that $d\phi^j$, $d\overline{\phi}^j\neq 0$ if, and only if, $3n-1\leq j\leq 4n-2$.

$(i)$ 
Even though it is well-known that on non-toral nilmanifolds there are no $1$-Kahler forms, since they coincide with K\"ahle metrics, we will consider the case $p=1$ for the benefit of the following constructions. We must construct a non closed $(8n-7)$-form satisfying property (\ref{eq:hind_medori_tomassini}). In particular, if we start from the $(8n-4)$-form
\[
\phi^{1}\wedge\dots\wedge\phi^{4n-2}\wedge\overline{\phi}^1\wedge\dots\wedge\overline{\phi}^{4n-2},
\]
we must remove three $1$-forms. For this purpose, we select $\phi^{2n-2}$, $\overline{\phi}^{3n-2}$, and $\overline{\phi}^{4n-2}$, therefore considering the $(4n-3,4n-4)$-form $\eta_1$ given by
\[
\eta_1=\phi^1\wedge\dots\wedge\hat{\phi^{2n-2}}\wedge\dots\wedge\phi^{4n-2}\wedge\overline{\phi}^1\wedge\dots\wedge\hat{\overline{\phi}^{3n-2}}\wedge\dots\wedge\overline{\phi}^{4n-3}.
\]
We now compute the $(4n-3,4n-3)$-component of $d\eta_1$. By the structure equations (\ref{eq:struct_eq_phi_delbar}), we remark that the only non trivial  relevant differentials are
\begin{align*}
&\delbar\phi^{3n-1}=\phi^{2n}\wedge\overline{\phi}^n,\\
&\delbar\phi^j=\phi^{j-2n}\wedge\overline{\phi}^{j-n},\quad 3n\leq j\leq 4n-2.
\end{align*}
In order to have a non vanishing term, we must ensure that $\delbar\phi^j=\phi^{2n-2}\wedge\overline{\phi}^{3n-2}$. However, this can happen if and only if $j=4n-2$, resulting in
\begin{align*}
d\eta_1^{(4n-3,4n-3)}&=d\left(\phi^1\wedge\dots\wedge\hat{\phi^{2n-2}}\wedge\dots\wedge\phi^{4n-2}\wedge\overline{\phi}^1\wedge\dots\wedge\hat{\overline{\phi}^{3n-2}}\wedge\dots\wedge\overline{\phi}^{4n-3}\right)\\
&=\phi^{1}\wedge\dots\wedge\phi^{4n-3}\wedge\overline{\phi}^1\wedge\dots\wedge\overline{\phi}^{4n-3}.
\end{align*}
Thus, considering $\psi_1:=\phi^1\wedge\dots\wedge^{4n-3}\in A^{4n-3,0}(X^{4n-2})$, we can conclude by Lemma \ref{lemma:hind-medori-tomassini}.\\
Therefore, for the case $1< p < n$, we can construct $\eta_p$ starting from the $(8n-7)$-form $\eta_1$ and then remove the forms $\phi^{3n-1}, \phi^{3n},\dots\phi^{3n+p-3},\overline{\phi}^{3n-1},\overline{\phi}^{3n},\dots,\overline{\phi}^{3n+p-3}$, (which accounts to removing $2p-2$ forms), obtaining a  non closed $(8n-2p-5)$-form. Then, the $(4n-2-p,4n-2-p)$-component of $d\eta_p$ is of type
\[
\psi_p\wedge\overline{\psi}_p,
\]
with $\psi_p\in A^{4n-2-p,0}(X^{4n-2})$  given by
\[
\psi_p=\phi^1\wedge\dots\wedge\phi^{3n-2}\wedge\phi^{3n+p-2}\wedge\dots\wedge\phi^{4n-3}.
\]
Again, we can conclude by Lemma \ref{lemma:hind-medori-tomassini}.\\
$(ii)$ Let us now consider the case $n\leq p \leq 4n-2$, starting from $p=n$ for the benefit of the following construction.\\
We must find a $(6n-5)$-form $\eta_n$ such that the $(3n-2,3n-2)$-component of $d\eta_n$ satisfies condition (\ref{eq:hind_medori_tomassini}). We construct the form $\eta_n$ as we have previously done, setting
\[
\eta_n=\phi^1\wedge\dots\wedge\phi^{\hat{2n-2}}\wedge\dots\wedge\phi^{3n-2}\wedge\phi^{4n-2}\wedge\overline{\phi}^1\wedge\dots\wedge\overline{\phi}^{3n-3},
\]
with $\eta_n\in A^{3n-2,3n-3}(X^{4n-2})$. By structure equations, we see that we have removed all the forms with non trivial differential but $d\phi^{4n-2}$. Therefore, when computing the differential $d\eta_n$, we obtain
\[
d\eta_n=-\phi^1\wedge\dots\wedge\phi^{3n-2}\wedge\overline{\phi}^1\wedge\dots\wedge\overline{\phi}^{3n-2}.
\]
By setting $\psi_n:=\phi^1\wedge\dots\wedge\phi^{3n-2}$, we conclude by Lemma \ref{lemma:hind-medori-tomassini}.\\
Now, if $n+1\leq p<4n-3$, we construct the $(8n-2p-5)$-form $\eta_p$ starting from the $(6n-5)$-form $\eta_n$ and then removing the forms $\phi^1,\dots,\phi^{p-n},\overline{\phi}^1,\dots,\overline{\phi}^{p-n}$. We clarify that, for $p-n\geq 2n-2$, since $\phi^{2n-2}$ has already been removed, we keep removing the $(1,0)$-forms with higher index starting from $\phi^{2n-1}$, whereas we keep $\overline{\phi}^{2n-2}$ and remove $\overline{\phi}^{2n-1}$ and so forth, so to remove $(1,0)$-forms for a total of $p-n$ forms and $(0,1)$-forms for a total of $p-n$ forms. This procedure accounts to building $\eta_p\in A^{4n-p-2,4n-p-3}$ as
\begin{align*}
\eta_{p}=
\phi^{p-n+1}\wedge\dots\wedge\phi^{\hat{2n-2}}\wedge\dots\wedge\phi^{3n-2}\wedge\phi^{4n-2}\wedge\overline{\phi}^{p-n+1}\wedge\dots\wedge\overline{\phi}^{3n-3}
\end{align*}
if $n+1\leq p <3n-3$, and
\begin{align*}
\eta_{p}=\phi^{p-n+2}\wedge\dots\wedge\phi^{3n-2}\wedge\phi^{4n-2}\wedge\overline{\phi}^{2n-2}\wedge\overline{\phi}^{p-n+1}\wedge\dots\wedge\overline{\phi}^{3n-3},
\end{align*}
if $3n-3\leq p\leq 4n-4$. We then compute $d\eta_{p}$. Since the only non trivial differential is $d\phi^{4n-2}$, we obtain
\[
d\eta_{p}=\epsilon_p\phi^{p-n+1}\wedge\dots\wedge\dots\wedge\phi^{3n-2}\wedge\overline{\phi}^{p-n+2}\wedge\dots\wedge\overline{\phi}^{3n-2}
\]
if $n+1\leq p <3n-3$, and
\[
d\eta_p=\epsilon_p\phi^{2n-2}\wedge\phi^{p-n+1}\wedge\dots\wedge\phi^{3n-2}\wedge\overline{\phi}^{2n-2}\wedge\overline{\phi}^{p-n+2}\wedge\dots\wedge\overline{\phi}^{3n-3},
\]
if $3n-3\leq p\leq 4n-4$. The number $\epsilon_p\in\{\pm 1\}$ is a sign term.  Therefore, by setting $$\psi_p=\phi^{p-n+1}\wedge\dots\wedge\dots\wedge\phi^{3n-2}$$ for $n+1\leq p <3n-3$ and $$\psi_p=\phi^{2n-2}\wedge\phi^{p-n+1}\wedge\dots\wedge\phi^{3n-2},$$ if $3n-3\leq p\leq 4n-4$, we can finally conclude by Lemma \ref{lemma:hind-medori-tomassini}.
\end{proof}

However, we show that there exist $(4n-3)$-K\"ahler forms. More precisely, we prove the following

\begin{theorem}\label{thm:main}
For every $n\geq 2$ the Bigalke and Rollenske's nilmanifold $X^{4n-2}$ admits balanced metrics.
\end{theorem}
\begin{proof}
%
We show that the diagonal Hermitian metric
$$
\omega:=\frac{i}{2}
\sum_{j=1}^{4n-2}\phi^j\wedge\overline{\phi}^j
$$
is balanced, i.e., $d\omega^{4n-3}=0$. Notice that
\[
\omega^{4n-3}=\left(\frac{i}{2}\right)^{4n-3}\frac{1}{(4n-3)!}\sum_{k=1}^{4n-2}\phi^{1}\wedge\overline{\phi}^1\wedge\dots\wedge\hat{\phi}^k\wedge\hat{\overline{\phi}}^k\wedge\dots\wedge \phi^{4n-2}\wedge\overline{\phi}^{4n-2}.
\]
We denote by $\alpha_k:=\phi^{1}\wedge\overline{\phi}^1\wedge\dots\wedge\hat{\phi}^k\wedge\hat{\overline{\phi}}^k\wedge\dots\wedge \phi^{4n-2}\wedge\overline{\phi}^{4n-2}$.
From the structure equations, when we compute $d\omega^{4n-3}$ we consider separately each term
\begin{equation*}
d\alpha_k=d(\phi^{1}\wedge\overline{\phi}^1\wedge\dots\wedge\hat{\phi}^k\wedge\hat{\overline{\phi}}^k\wedge\dots\wedge \phi^{4n-2}\wedge\overline{\phi}^{4n-2}).
\end{equation*}
By the structure equations we have that $d\alpha_k=0$ for every $k=1,\dots,4n-2$. Indeed, by Leibniz rule, the only way to have $d\alpha^k\neq 0$ would be that for some index $j=1,\dots,\hat{k},\dots,4n-2$, $d\phi^j$ or $d\bar\phi^j$ contains exactly $\phi^k\wedge\bar\phi^k$. But this is not the case as showed by the structure equations. Hence, $d\alpha_k=0$ for every $k=1,\dots,4n-2$ and so
$d\omega^{4n-3}=0$ and so $\omega$ is balanced.\\
\end{proof}

As a consequence, combining this with \cite[Theorem 1]{bigalke-rollenske}, we get that there is no relation between the existence of balanced metrics and the degeneracy step of the Fr\"olicher spectral sequence.
\begin{cor}\label{cor:balanced-degeneration}
On balanced manifolds the degeneracy step of the Fr\"olicher spectral sequence can be arbitrarily large.
\end{cor}
In particular, this is in contrast with the situation in K\"ahler geometry where for compact K\"ahler manifolds the Fr\"olicher spectral sequence degenerates at the first step and with a conjecture by Popovici stating that on compact SKT manifolds the Fr\"olicher spectral sequence degenerates at the second step (cf. \cite[Conjecture 1.3]{popovici}). For completeness, we recall that an \emph{SKT} (or \emph{pluriclosed}) metric on a complex manifold is an Hermitian metric $\omega$ such that $\del\delbar\omega=0$.\\
In fact, in relation with this conjecture we show explicitly the following
\begin{prop}
For every $n\geq 2$ the Bigalke and Rollenske's nilmanifold $X^{4n-2}$ does not admit any SKT metric.
\end{prop}
\begin{proof}
In order to show that $X^{4n-2}$ does not admit any SKT metric we use the characterization of \cite{egidi} in terms of currents. More precisely, we will construct a non-zero positive $(1,1)$-current which is $\del\delbar$-exact.\\
Indeed, by a direct computation using the structure equations
$$
\psi:=\phi^1\wedge\bar\phi^1\wedge\dots\wedge \phi^{4n-3}\wedge\bar\phi^{4n-3}
=
$$
$$
\del\delbar\left(
\phi^1\wedge\bar\phi^1\wedge\dots\wedge
\hat{\phi}^{n-1}\wedge\hat{\bar\phi}^{n-1}\wedge\dots\wedge
\hat{\phi}^{2n-1}\wedge\hat{\bar\phi}^{2n-1}\wedge\dots\wedge
\phi^{4n-2}\wedge\bar\phi^{4n-2}
\right)\,.
$$
The $(4n-3,4n-3)$-form $\psi$ gives rise to a $\del\delbar$-exact non-zero positive $(1,1)$-current on $X$.
\end{proof}
Notice that this follows also by \cite{fino-vezzoni} where the authors show that on non-tori nilmanifolds balanced and SKT metrics cannot coexist.\\
We recall that an Hermitian metric $\omega$ on a complex manifold is called \emph{locally conformally K\"ahler} if
$$
d\omega=\theta\wedge\omega
$$
where $\theta$ is a $d$-closed $1$-form.
We then notice also the following
\begin{prop}
For every $n\geq 2$ the Bigalke and Rollenske's nilmanifold $X^{4n-2}$ does not admit any locally conformally K\"ahler metric.
\end{prop}
\begin{proof}
This follows directly combining Theorem \ref{thm:main} and \cite[Theorem 4.9]{ornea-otiman-stanciu} where it is proved that on non-tori complex nilmanifolds endowed with a left-invariant complex structure, locally conformally K\"ahler metrics and balanced metrics cannot coexist.
\end{proof}

\end{document}